\documentclass{article}%
\usepackage{authblk}
\usepackage{amsmath}

\usepackage{amsfonts}
\usepackage{amssymb,bm}
\usepackage{geometry}
\usepackage{scalefnt}
\usepackage{xcolor}
\usepackage[colorlinks=true,linkcolor=red,citecolor=blue,urlcolor=cyan]{hyperref}
\usepackage{graphicx,epstopdf}
\usepackage{subfigure}
\usepackage{fullpage}
\usepackage{float}
\providecommand{\U}[1]{\protect\rule{.1in}{.1in}}
\newtheorem{theorem}{Theorem}

\newenvironment{proof}[1][Proof]{\noindent\textbf{#1.} }{\ \rule{0.5em}{0.5em}}
\begin{document}
	\title{Exact solution of the (2+1)-dimensional damping forcing  coupled Burgers equation by using   Darboux transformation}
	\author[1] {Prasanta Chatterjee\thanks{\href{mailto:@gmail.com}{@gmail.com}}}
	\author[1] {Nanda Kanan Pal\thanks{\href{mailto:@gmail.com}{@gmail.com}}}
	\author[1] {Dipan Saha\thanks{\href{mailto:@gmail.com}{@gmail.com}}}
	\author[2] {SANTANU RAUT\thanks{Present address: \href{mailto:raut\_santanu@yahoo.com}{raut\_santanu@yahoo.com}}}
	\affil[1]{Department of Mathematics, Visva-Bharati University, Santiniketan, 731235, India  India}
	\affil[2]{ Department of Mathematics, Mathabhanga College, Cooch Behar, West Bengal-736164, India.}
	%	\affil[3]{Department of Mathematics, Zhejiang Normal University, Jinhua 321004, Zhejiang, China}
	%	\affil[4]{Department of Mathematics, King Abdulaziz University, Jeddah 21589, Saudi Arabia}
	%	\affil[5]{Department of Mathematics and Statistics, University of South Florida, Tampa, FL 33620-5700, USA}
	%	\affil[6]{School of Mathematical and Statistical Sciences, North-West University, Mafikeng Campus, Private Bag X2046, Mmabatho 2735, South Africa}
	\date{}
	\maketitle
	\begin{abstract}
	In this article, we investigate the (2+1)-dimensional damping forcing coupled Burgers equation, which is obtain by adding damping and forcing terms from couple Burgers equation. The Lax pair of the (2+1)-dimensional damping forcing coupled Burgers equation is establised. With the help of Lax pair, we derive the $N$-fold Darboux transformation of (2+1)-dimensional damping forcing coupled Burgers equation. Using one fold and two fold Darboux transformation, we demostrated some wave solutions including solitary wave solution and periodic wave solution. The impact of damping and forcing terms in solitary wave solution and periodic solution is graphically demostrated. 
	\end{abstract}
	\section{Introduction}
Integrable or non-integrable linear and nonlinear differential equations are very essential to study the various features of nonlinear physical phenomena in real life problems. To study the nonlinear physical phenomena, the exact solutions of the linear and nonlinear  differential equations plays an important role in various branch of science as physics, chemistry, biology, engineering and economics etc. The theory is well developed to construct the exact solutions of the linear and nonlinear differential equations by various analytic methods as Darboux transformation \cite{darbox_1}, inverse scattering method \cite{inv1}-\cite{AKNS}, Hirota bilinear method \cite{hirota1}, algebro-geometric method \cite{algeb1}-\cite{algeb2}, and Lie group method \cite{sk}, etc. The (2+1)-dimensional damping forcing coupled Burgers (DFCB) equation read as
\begin{equation}
\begin{split}
&u_{t}-2u_{y}+\frac{3}{2}u_{xx}-3v_{x}-S(t)u=T(t),\\
&v_{t}-2v_{y}-u_{y}+u_{xxx}+P(t)uu_x-\frac{3}{2}v_{xx}+R(t)u_x-S(t)v=T(t),\label{dfcb}
\end{split}
\end{equation}
where $u=u(x,y,t), v=v(x,y,t)$, $P(t)=\frac{1}{\Lambda(t)}$, $R(t)=-\frac{H(t)}{\Lambda(t)}$, $S(t)=\frac{\Lambda'(t)}{\Lambda(t)}$ and $T(t)=H'(t)-\frac{\Lambda'(t)}{\Lambda(t)}H(t)$.\\
The (2+1)-dimensional DFCB equation \eqref{dfcb} was modelled by introducing damping term $ S(t)u$ and forcing term $T(t)$ from couple Burgers equation \cite{Tingsu}. The main object in this paper  is to study the (2+1)-dimensional damping forcing coupled Burgers equation by using Darboux tranformation and Lax pair. The existence of Lax pair of a linear or nonlinear differential equation ensure that the equation is integrable in the sence of Lax pair \cite{P.D.Lax}. The Darboux transformation is a powerful tool to construct various prestigious solutions such as solitary wave solution, periodic solution, positon solution, negaton solution and complexiton solution, etc. of the nonlinear differential equation \cite{nk1}-\cite{pc}.\\
The present article is organized as follows. In section 2, we construct the Lax pair of (2+1)-dimensional damping forcing coupled Burgers equation and confirmed its intigribility. In section 3, we derived the $N$- fold Darboux transformation of (2+1)-dimensional damping forcing coupled Burgers equation via its Lax pair. Finally in section 4, we construct some wave solutions especally solitary wave solution and periodic wave solution by using one fold and two fold Darboux transformation.
\section{Lax pair of damping forcing couple Burger equation}
	In this section, we demostrated the Lax pair of the DFCB equation. The Lax pair of the DFCB equation defined as
\begin{equation}
L\psi=\psi_y,\hspace{5mm}\psi_t=M\psi,\label{laxpair}
\end{equation}
where
\begin{equation}
\begin{split}
&L=\partial_x^3+\frac{u-H(t)}{\Lambda(t)}\partial_x+\frac{v-H(t)}{\Lambda(t)}I,~~ M=\partial_x^3+\frac{3}{2}\partial^2_x+\frac{u-H(t)}{\Lambda(t)}\partial_x+\partial_y+\frac{u+v-2H(t)}{\Lambda(t)}I.
\end{split}
\end{equation}
Here $L$ and $M$ is called the Lax operators and $I$ read as the identity operator. By the compatibilty conditions $\psi_{xxt}=\psi_{txx}$ and , the equations \eqref{laxpair} reduced to zero curvatutre equation $(L_t-M_y+LM-ML)=0$, which is equivalent to DFCB equation \eqref{dfcb}. Thus the DFCB equation \eqref{dfcb} is integrable through Lax pair.
\section{Darboux transformation}
Darboux transformation takes an important role in building multi-soliton solutions via the old solutions of integrable systems. From a given linear system of integrable systems, these transformations are inferred. We first go over a few theorems and propositions connected to the current research before attempting to deduce the solutions to the DFCB equation.
\begin{theorem}
	Suppose $\psi_{1}$ is a non zero particular solution of Lax pair \eqref{laxpair}. Then  $\psi[1]=(\partial_x-\partial_x\log\psi_1)\psi$ is the new eigen function which obeys the Lax pair $L[1]\psi[1]=\psi_y[1] $,\quad $\Psi_t[1]=M[1]\psi[1]$ corresponding to the potential,  
	\begin{equation}
	u[1]=u+3\Lambda(t)\partial^{2}_x\log\psi_1,\hspace{3mm} v[1]=v+u_x+3\Lambda(t)(\partial_x\log\psi_1\partial^{2}_x\log\psi_1+\partial^{3}_x\log\psi_1).\label{Ns}
	\end{equation}
	where 
\begin{equation}
\begin{split}
&L[1]=\partial_x^3+\frac{u[1]-H(t)}{\Lambda(t)}\partial_x+\frac{v[1]-H(t)}{\Lambda(t)}I,~~ M[1]=\partial_x^3+\frac{3}{2}\partial^2_x+\frac{u[1]-H(t)}{\Lambda(t)}\partial_x+\partial_y+\frac{u[1]+v[1]-2H(t)}{\Lambda(t)}I.
\end{split}
\end{equation}	
	
\end{theorem}
\begin{proof}
Let the linear tranformation that is called Darboux transformation defined on $\psi$ in \eqref{laxpair} by
\begin{equation}
	\psi[1]=(\partial_x-\partial_x\log\psi_1)\psi=\psi_{x}+B\psi, \label{DT}
\end{equation}
where, $B=-\partial_x\log\psi_1$ and $\psi_{1}$ is a non zero particular solution of Lax pair \eqref{laxpair}.
From \eqref{DT}, we have
\begin{equation}
\begin{split}
&\psi[1]=\psi_x+B\psi,~\psi_x[1]=\psi_{xx}+B\psi_x+B_x\psi,~\psi_{xx}[1]=\psi_{xxx}+B\psi_{xx}+2B_x\psi_x+B_{xx}\psi,\\&\psi_{xxx}[1]=\psi_{yx}+(3B_x-\frac{u-H(t)}{\Lambda(t)})\psi_{xx}+(3B_{xx}-\frac{v-H(t)}{\Lambda(t)}-\frac{u_x}{\Lambda(t)}-\frac{u-H(t)}{\Lambda(t)}B)\psi_x\\&+(B_{xxx}-\frac{v-H(t)}{\Lambda(t)}B-\frac{v_x}{\Lambda(t)})\psi+B\psi_y,~~
\psi_y[1]=\psi_{xy}+B\psi_y+B_y\psi.
\end{split}
\end{equation}
We construct two expressions $\Delta$ and $\Delta'$ such that
\begin{equation}
\Delta=L[1]\psi[1]-\psi_y[1],\hspace{5mm}\Delta'=\psi_t[1]-M[1]\psi[1],\label{del1}
\end{equation}
where 
\begin{equation}
\begin{split}
&L[1]=\partial_x^3+\frac{u[1]-H(t)}{\Lambda(t)}\partial_x+\frac{v[1]-H(t)}{\Lambda(t)}I,~ M[1]=\partial_x^3+\frac{3}{2}\partial^2_x+\frac{u[1]-H(t)}{\Lambda(t)}\partial_x+\partial_y+\frac{u[1]+v[1]-2H(t)}{\Lambda(t)}I.
\end{split}
\end{equation}
Substituting the values of $\psi[1]$ and its spatial derivatives $\psi_x[1]$,  $\psi_{xx}[1]$, $\psi_{xxx}[1]$, and $\psi_{y}[1]$, into equation \eqref{del1}, we have
\begin{equation}
\Delta=\Delta_1\psi_{xx}+\Delta_2\psi_{x}+\Delta_3\psi, \quad \Delta'=\Delta'_1\psi_{xx}+\Delta'_2\psi_{x}+\Delta'_3\psi,
\end{equation}
where
\begin{equation}
\begin{split}
&\Delta_1=(3B_x-\frac{u-H(t)}{\Lambda(t)}+\frac{u[1]-H(t)}{\Lambda(t)}),~
\Delta_2=(3B_{xx}-\frac{v-H(t)}{\Lambda(t)}-\frac{u_x}{\Lambda(t)}-\frac{u-H(t)}{\Lambda(t)}B+\frac{u[1]-H(t)}{\Lambda(t)}B+\frac{v[1]-H(t)}{\Lambda(t)}),\\
&\Delta_3=(B_{xxx}-\frac{v-H(t)}{\Lambda(t)}B-\frac{v_x}{\Lambda(t)}+\frac{u[1]-H(t)}{\Lambda(t)}B_x+\frac{v[1]-H(t)}{\Lambda(t)}B-B_y).
\end{split}
\end{equation}
and 
\begin{equation}
\begin{split}
&\Delta'_1=(\frac{u-H(t)}{\Lambda(t)}-3B_x-\frac{u[1]-H(t)}{\Lambda(t)}),\\
&\Delta'_2=\frac{v-H(t)}{\Lambda(t)}+\frac{u-H(t)}{\Lambda(t)}+\frac{u-H(t)}{\Lambda(t)}B-3B_x-\frac{u[1]-H(t)}{\Lambda(t)}B+\frac{u_x}{\Lambda(t)}-\frac{u[1]+v[1]-2H(t)}{\Lambda(t)},\\
&\Delta'_3=B_t+\frac{u-H(t)}{\Lambda(t)}B-B_{3x}+\frac{v-H(t)}{\Lambda(t)}B-\frac{3}{2}B_{2x}-\frac{u[1]-H(t)}{\Lambda(t)}B_x+\frac{u_x+v_x}{\Lambda(t)}-\frac{u[1]+v[1]-2H(t)}{\Lambda(t)}B-B_y.
\end{split}
\end{equation}
For new potential we consider $\Delta_1=0$ and $\Delta_2=0$, so the new potential becomes 
\begin{equation}
\begin{split}
&u[1]=u-3\Lambda(t)B_x=u+3\Lambda(t)\partial^{2}_x\log\psi_1,\\
&v[1]=v+u_x+3\Lambda(t)(BB_x-B_{xx})=v+u_x+3\Lambda(t)(\partial_x\log\psi_1\partial^{2}_x\log\psi_1+\partial^{3}_x\log\psi_1).
\end{split}
\end{equation}
Therefore the new Lax pair corresponding to new potentials $u[1]$ and
$v[1]$ becomes
\begin{equation}
 L[1]\psi[1]=\psi_y[1],\hspace{2mm}\psi_t[1]=M[1]\psi[1]. \label{nlaxpair}
\end{equation}
 Hence by the compatibilty condition, we can ensure that $u[1]$ and
$v[1]$ are the new exact solution of DFCB equation corresponding to old solution $u$ and $v$ of the DFCB equation.\\
This complete the proof.
\end{proof}
\begin{theorem}
	Let $\psi_{1}$, $\psi_{2}$, $\psi_{3}$, $\cdots$, $\psi_{N}$ be the non-zero particular solutions of the linear equations \eqref{laxpair}. Then the new eigen function  $\psi[N]=\frac{W(\psi_{1},\psi_{2},\psi_{3},\cdots,\psi_{N},\psi)}{W(\psi_{1},\psi_{2},\psi_{3},\cdots,\psi_{N})}$ satisfy the Lax pair  $L[N]\psi[N]=\psi_y[N] $,\quad $\psi_t[N]=M[N]\psi[N]$ corresponding to the potential stated below
\begin{equation}
\begin{split}
&u[N]=u+3\Lambda(t)\partial^{2}_{x}\log W(\psi_{1},\psi_{2},\psi_{3},\cdots,\psi_{N}),\\
&v[N]=v+Nu_x+3\Lambda(t)\sum_{r=1}^{N}\bigg(\partial_x\log \frac{W(\psi_1,\psi_2,\cdots,\psi_r)}{W(\psi_1,\psi_2,\cdots,\psi_{r-1})}\partial^{2}_x\log\frac{W(\psi_1,\psi_2,\cdots,\psi_r)}{W(\psi_1,\psi_2,\cdots,\psi_{r-1})}+\partial^{3}_x\log W(\psi_1,\psi_2,\cdots,\psi_r)\bigg)\label{Nsol}
\end{split}
\end{equation}
where
\begin{equation}
\begin{split}
&L[N]=\partial_x^3+\frac{u[N]-H(t)}{\Lambda(t)}\partial_x+\frac{v[N]-H(t)}{\Lambda(t)}I,~ M[N]=\partial_x^3+\frac{3}{2}\partial^2_x+\frac{u[N]-H(t)}{\Lambda(t)}\partial_x+\partial_y+\frac{u[N]+v[N]-2H(t)}{\Lambda(t)}I.
\end{split}
\end{equation}
and 	$W(\psi_{1},\psi_{2},\psi_{3},\cdots,\psi_{K})=\text{det}(\theta_1, \theta_2, \theta_3,\cdots, \theta_K)$, $\theta_j=(\psi_j,\partial_{x}\psi_j,\partial^2_{x}\psi_j,\cdots,\partial^{K-1}_{x}\psi_j)^T$ is the column vector forming by $\psi_j$ and its partial derivatives upto $(K-1)$ th order with respect to the variable $x$ , $j=1,2,3,\cdots, N$
\end{theorem}
\begin{proof}
We use the induction theory to prove the above result for $N$-fold Darboux transformation. For $N=1$, we proved in theorem 1, the result is true for one-fold Darboux transformation. Now we shall prove the result is true for $N=2$. Again we apply the one-fold Darboux transformation in equation \eqref{nlaxpair} defined as
\begin{equation}
\psi[2]=(\partial_x-\partial_x\log\psi_1[1])\psi[1] \label{2DT}
\end{equation}
Where $\psi_1[1]$, is the nonzero particular solution of \eqref{nlaxpair}. The value of $\psi_1[1]$, one can obtained by \eqref{DT} is $\psi_1[1]=\frac{W(\psi_{1},\psi_{2})}{\psi_{1}}$. Here $\psi_{2}$ is the non zero particular solution of \eqref{laxpair}. By equation \eqref{DT}, the expression of $\psi[2]$ in equation \eqref{2DT} can be written as in the following form
\begin{equation}
\psi[2]=\frac{W(\psi_{1},\psi_{2},\psi)}{W(\psi_{1},\psi_{2})},
\end{equation}
which is called 2-fold Darboux transformation on $\psi$ in \eqref{laxpair}. Then the new potentials becomes
\begin{equation}
\begin{split}
&u[2]=u[1]+3\Lambda(t)\partial^{2}_x\log\psi_1[1],~
v[2]=v[1]+u_x[1]+3\Lambda(t)(\partial_x\log\psi_1[1]\partial^{2}_x\log\psi_1[1]+\partial^{3}_x\log\psi_1[1]).
\end{split}
\end{equation}
Subtitutiong the values of $u[1]$, $v[1]$, and  $\psi_1[1]$ into above expression, we have
\begin{equation}
\begin{split}
&u[2]=u+3\Lambda(t)\partial^{2}_{x}\log W(\psi_{1},\psi_{2}),\\
&v[2]=v+2u_x+3\Lambda(t)\sum_{r=1}^{2}\bigg(\partial_x\log \frac{W(\psi_1,\psi_2,\cdots,\psi_r)}{W(\psi_1,\psi_2,\cdots,\psi_{r-1})}\partial^{2}_x\log\frac{W(\psi_1,\psi_2,\cdots,\psi_r)}{W(\psi_1,\psi_2,\cdots,\psi_{r-1})}+\partial^{3}_x\log W(\psi_1,\psi_2,\cdots,\psi_r)\bigg)
\end{split}
\end{equation}
Thus the result is true for $N=2$.\\
 We proved the Darboux transformation result is true for $N=1$ and $N=2$, we shall assume the result is true for $N=K$, therefore for $K$- fold Darboux transformation on $\psi$ in \eqref{laxpair} is	
\begin{equation}
\psi[K]=\frac{W(\psi_{1},\psi_{2},\psi_{3},...\psi_{K},\psi)}{W(\psi_{1},\psi_{2},\psi_{3},...\psi_{K})} \label{KDT}
\end{equation}
Thus $\psi[K]$ obey the followings equations
\begin{equation}
\begin{split}
 L[K]\psi[K]=\psi_y[K],\hspace{2mm}\psi_t[K]=M[K]\psi[K] 
 \label{covlaxpairNN}
\end{split}
\end{equation}
where the new potential becomes after $K$-fold Darboux transformation
\begin{equation}
 \begin{split}
 &u[K]=u+3\Lambda(t)\partial^{2}_{x}\log W(\psi_{1},\psi_{2},\psi_{3},\cdots,\psi_{K}),\\
 &v[K]=v+Ku_x+3\Lambda(t)\sum_{r=1}^{K}\bigg(\partial_x\log \frac{W(\psi_1,\psi_2,\cdots,\psi_r)}{W(\psi_1,\psi_2,\cdots,\psi_{r-1})}\partial^{2}_x\log\frac{W(\psi_1,\psi_2,\cdots,\psi_r)}{W(\psi_1,\psi_2,\cdots,\psi_{r-1})}+\partial^{3}_x\log W(\psi_1,\psi_2,\cdots,\psi_r)\bigg)
 \end{split}                              \label{newsolNN}
\end{equation}
and $W(\psi_{1},\psi_{2},\psi_{3},\cdots,\psi_{K})=\text{det}(\theta_1, \theta_2, \theta_3,\cdots, \theta_K)$, $\theta_j=(\psi_j,\partial_{x}\psi_j,\partial^2_{x}\psi_j,\cdots,\partial^{K-1}_{x}\psi_j)^T$ is the column vector forming by $\psi_j$ and its partial derivatives upto $(K-1)$ th order with respect to the variable $x$ , $j=1,2,3,\cdots, N$ and $\psi_{1}$, $\psi_{2}$, $\psi_{3}$, $\cdots$, $\psi_{K}$ are the nonzero solutions of \eqref{laxpair}.\\
The one-fold Darboux transformation on $\psi[K]$ in \eqref{covlaxpairNN} defined as
\begin{equation}
\psi[K+1]=(\partial_x-\partial_x\log\psi_K[K])\psi[K],\label{DTK}
\end{equation}
where $\psi_K[K]$ is the non zero particular solution of \eqref{covlaxpairNN}. Thus the value of $\psi_K[K]$ one can obtained through \eqref{KDT} as
\begin{equation}
\psi_K[K]=\frac{W(\psi_{1},\psi_{2},\psi_{3},...,\psi_{K},\psi_{K+1})}{W(\psi_{1},\psi_{2},\psi_{3},...,\psi_{K})}
\end{equation}
where $\Psi_{K+1}$ be the nonzero particular solution of \eqref{laxpair}. Substituing the values of  $\psi[K]$ and $\psi_K[K]$  into \eqref{DTK}, thus we obtained $(K+1)$-fold Darboux transformation on $\Psi$ as 
\begin{equation}
\Psi[K+1]=\frac{W(\psi_{1},\psi_{2},\psi_{3},...,\psi_{K+1},\psi)}{W(\psi_{1},\psi_{2},\psi_{3},...\psi_{K+1})} \label{MKDT}
\end{equation}
Therefore $\psi[K+1]$ obey the following equations
\begin{equation}
\begin{split}
L[K+1]\psi[K+1]=\psi_y[K+1],\hspace{2mm}\psi_t[K+1]=M[K+1]\psi[K+1] 
 \label{Mcovlaxpair}
\end{split}
\end{equation}
thus the new potential becomes after (K+1)-fold Darboux transformation as
\begin{equation}
\begin{split}
u[K+1]=&u[K]+3\Lambda(t)\partial^{2}_{x}\log (\psi_K[K]),  \\
=&u+3\Lambda(t)\partial^{2}_{x}\log W(\psi_{1},\psi_{2},\psi_{3},\cdots,\psi_{K})\\&+3\Lambda(t)\partial^{2}_{x}\log\bigg(\frac{W(\psi_{1},\psi_{2},\psi_{3},...\psi_{K},\psi_{K+1})}{W(\psi_{1},\psi_{2},\psi_{3},...\psi_{K})}\bigg),\\
=&u+3\Lambda(t)\partial^{2}_{x}\log W(\psi_{1},\psi_{2},\psi_{3},\cdots,\psi_{K+1}).
\end{split}                                             \label{Knewsol}
\end{equation}
and 
\begin{equation}
\begin{split}
&v[K+1]=v[K]+u_x[K]+ 3\Lambda(t)(\partial_x\log \psi_K[K]\partial^{2}_x\log \psi_K[K]+\partial^{3}_x\log \psi_K[K])
\end{split}
\end{equation}
By using the values of $v[K]$, $u[K]$, and $\psi_K[K]$, the above expression reduced to 
\begin{equation}
v[K+1]=v+(K+1)u_x+3\Lambda(t)\sum_{r=1}^{K+1}\bigg(\partial_x\log \frac{W(\psi_1,\psi_2,\cdots,\psi_r)}{W(\psi_1,\psi_2,\cdots,\psi_{r-1})}\partial^{2}_x\log\frac{W(\psi_1,\psi_2,\cdots,\psi_r)}{W(\psi_1,\psi_2,\cdots,\psi_{r-1})}+\partial^{3}_x\log W(\psi_1,\psi_2,\cdots,\psi_r)\bigg)
\end{equation}
Thus the result is true for $N=K+1$.\\
This complete the proof.
\end{proof}
\section{Solution of the damped forced Burger's equation}
\subsection{Solution using one fold Darboux transformation}
It is easy to check that $u(x,y,t)=H(t)$, $v(x,y,t)=H(t)$ is trivial  solution of the DFCB  equations \eqref{dfcb}. Substituting the trivial solutions $u(x,y,t)=H(t)$ and $v(x,y,t)=H(t)$ into linear equations \eqref{laxpair}, then we solve the systems for nonzero solution $\psi_{1}$. Thus we have
\begin{equation}
\psi_{1}(x,y,t)=c_1\text{e}^{\xi_1}+c_2\text{e}^{\xi_2}\text{sin}(\xi_3)+c_3\text{e}^{\xi_2}\text{cos}(\xi_3),
\end{equation}
where $\xi_1=k_1x+k_1^3y+(2k_1^3+\frac{3}{2}k_1^2)t,~ \xi_2=k_1x-2k_1^3y-4k_1^3t,$ and $\xi_3=k_1x+2k_1^3y+(4k_1^3+3k_1^2)t$. Here $c_1, ~c_2,~c_3 $ and $k_1$ are arbitrary real value. Substituting $\psi_1$ into equation \eqref{Ns} and labeled, we have the new nontrivial solution of the DFCB equation as
\begin{equation}
\begin{split}
&u[1]=H(t)-\frac{3\Lambda(t) \Theta_1}{\Theta^2},~~
v[1]=H(t)-\frac{3\Lambda(t)\Theta_2}{2\Theta^3}
\end{split}
\end{equation}
where
\begin{equation*}
\begin{split}
&\Theta_2=k_1^3\text{e}^{d_2}\bigg [  3c_1c_2^2\text{e}^{d_1+d_2}+3c_1c_3^2\text{e}^{d_1+d_2}+2\bigg ( c_1^2(c_2+c_3)\text{e}^{2d_1}+(c_3^3-c_2^3+c_2^2c_3-c_2c_3^2)\text{e}^{2d_2}\bigg)\cos\xi_3\\
&~~~~-(c_1c_2^2-c_1c_3^2)\text{e}^{d_1+d_2}\cos 2\xi_3+2(c_1^2c_2-c_1^2c_3)\text{e}^{2d_1}\sin \xi_3+2c_2^3\text{e}^{2d_2}\sin \xi_3+2c_2^2c_3\text{e}^{2d_2}\sin \xi_3\\&~~~~+2c_2c_3^2\text{e}^{2d_2}\sin \xi_3+2c_3^3\text{e}^{2d_2}\sin \xi_3+2c_1c_2c_3\text{e}^{d_1+d_2}\sin 2\xi_3\bigg],
~\Theta=(c_1\text{e}^{d_1}+c_2\text{e}^{d_2}\sin \xi_3+c_3\text{e}^{d_2}\cos \xi_3),\\
&\Theta_1=k_1^2\text{e}^{d_2}\bigg[(c_2^2+c_3^2)\text{e}^{d_2}+c_1c_3\text{e}^{d_1}\cos\xi_3+c_1c_2\text{e}^{d_1}\cos\xi_3\bigg],
~d_1=k_1^3y+(2k_1^3+\frac{3}{2}k_1^2)t,
~d_2=-2k_1^3y-4k_1^3t.
\end{split}
\end{equation*}
\subsection{Solution using two fold Darboux transformation}
In this section we construct wave solution of  of the DFCB  equations \eqref{dfcb} through by two-folded Darboux transformation with intial solution trivial solution $u(x,y,t)=H(t)$, $v(x,y,t)=H(t)$. Substituting the trivial solutions $u(x,y,t)=H(t)$ and $v(x,y,t)=H(t)$ into linear equations \eqref{laxpair}, then we solve the systems for nonzero solution $\psi_{1}$ and $\psi_{2}$. Thus we have
\begin{equation}
\begin{split}
&\psi_{1}(x,y,t)=c_1\text{e}^{\xi_1}+c_2\text{e}^{\xi_2}\text{sin}(\xi_3)+c_3\text{e}^{\xi_2}\text{cos}(\xi_3),~
\psi_{2}(x,y,t)=p_1\text{e}^{\delta_1}+p_2\text{e}^{\delta_2}\text{sin}(\delta_3)+p_3\text{e}^{\delta_2}\text{cos}(\delta_3)\\
\end{split}
\end{equation}
where $\delta_1=k_2x+k_2^3y+(2k_2^3+\frac{3}{2}k_2^2)t,~ \delta_2=k_2x-2k_2^3y-4k_2^3t,$ and $\delta_3=k_2x+2k_2^3y+(4k_2^3+3k_2^2)t$. Here $p_1, ~p_2,~p_3 $ and $k_2$ are arbitrary real value. For $N=2$, through equation \eqref{Nsol} the expession of new solution $u[2]$ and $v[2]$ after two fold Darboux transformation become
\begin{equation}
\begin{split}
&u[2]=u+3\Lambda(t)\partial^{2}_{x}\log W(\psi_{1},\psi_{2}),\\
&v[2]=v+2u_x+3\Lambda(t)\bigg(\partial_x\log\psi_1\partial^{2}_x\log\psi_1+\partial^{3}_x\log\psi_1  +\partial_x\log \frac{W(\psi_1,\psi_2)}{\psi_1}\partial^{2}_x\log\frac{W(\psi_1,\psi_2)}{\psi_1}+\partial^{3}_x\log W(\psi_1,\psi_2\bigg)\label{2S}
\end{split}
\end{equation}
Substituting $\psi_1$ and $\psi_2$ into the above equation \eqref{2S} and labeled, we have the new nontrivial solution of the DFCB equation as
\begin{equation}
\begin{split}
&u[2]=H(t)+\frac{3\Lambda(t)\Theta_3}{\Theta_4},~~
v[2]=H(t)+\frac{3\Lambda(t)\Gamma}{\Theta^3\Theta_3\Theta_4^2\Theta_6}
\end{split}
\end{equation}
where
\begin{equation}
\begin{split}
\Gamma&=\Theta\Theta_3^2\Theta_4\Theta_5-\Theta^2\Theta_3^2\Theta_4\Theta_6\Theta_7+\Theta\Theta_1\Theta_3\Theta_4\Theta_4^2\Theta_5-\Theta_1\Theta_3\Theta_4^2\Theta_6\Theta_7\\&-\Theta^3\Theta_2\Theta_4^2\Theta_6+\Theta^3\Theta_3\Theta_4\Theta_6\Theta_{3,x}-\Theta^3\Theta_3^2\Theta_6\Theta_{4,x},\\
\Theta_3&=-\bigg[ k_1^2(c_1\text{e}^{d_1}+2c_2\text{e}^{d_2}\cos\xi_3-2c_3\text{e}^{d_2}\sin\xi_3)(p_1\text{e}^{r_1}+p_2\text{e}^{r_2}\sin\delta_3+p_3\text{e}^{r_2}\cos\delta_3)\\
&-k_2^2(c_1\text{e}^{d_1}+c_3\text{e}^{d_2}\cos\xi_3+c_2\text{e}^{d_2}\sin\xi_3)(p_1\text{e}^{r_1}+2p_2\text{e}^{r_2}\cos\delta_3-2p_3\text{e}^{r_2}\sin\delta_3)\bigg]^2\\
&+\bigg[-k_1(c_1\text{e}^{d_1}+(c_2+c_3)\text{e}^{d_2}\cos\xi_3+(c_2-c_3)\text{e}^{d_2}\sin\xi_3)(p_1\text{e}^{r_1}+p_2\text{e}^{r_2}\sin\delta_3+p_3\text{e}^{r_2}\cos\delta_3)\\
&+k_2(c_1\text{e}^{d_1}+c_3\text{e}^{d_2}\cos\xi_3+c_2\text{e}^{d_2}\sin\xi_3)(p_1\text{e}^{r_1}
+p_2\text{e}^{r_2}(\cos\delta_3+\sin\delta_3) +p_3\text{e}^{r_2}(\cos\delta_3-\sin\delta_3))\bigg]\\
&\times\bigg[-k_1^3(c_1\text{e}^{d_1}+2(c_2-c_3)\text{e}^{d_2}\cos\xi_3-2(c_2+c_3)\text{e}^{d_2}\sin\xi_3) (p_1\text{e}^{r_1}+p_2\text{e}^{r_2}\sin\delta_3+p_3\text{e}^{r_2}\cos\delta_3)\\
&-k_1^2k_2(c_1\text{e}^{d_1}+2c_2\text{e}^{d_2}\cos\xi_3-2c_3\text{e}^{d_2}\sin\xi_3)(p_1\text{e}^{r_1} +p_2\text{e}^{r_2}(\cos\delta_3+\sin\delta_3) +p_3\text{e}^{r_2}(\cos\delta_3-\sin\delta_3))\\
&+k_1k_2^2(c_1\text{e}^{d_1}+(c_2+c_3)\text{e}^{d_2}\cos\xi_3+(c_2-c_3)\text{e}^{d_2}\sin\xi_3)(p_1\text{e}^{r_1}+2p_2\text{e}^{r_2}\cos\delta_3-2p_3\text{e}^{r_2}\sin\delta_3)\\
&+k_2^3(c_1\text{e}^{d_1}+c_3\text{e}^{d_2}\cos\xi_3+c_2\text{e}^{d_2}\sin\xi_3)(p_1\text{e}^{r_1} +2p_2\text{e}^{r_2}(\cos\delta_3-\sin\delta_3) -2p_3\text{e}^{r_2}(\cos\delta_3+\sin\delta_3))\bigg],\\
\end{split}
\end{equation}
\begin{equation}
\begin{split}
&\Theta_4=\bigg[ k_1(c_1\text{e}^{d_1}+(c_2+c_3)\text{e}^{d_2}\cos\xi_3+(c_2-c_3)\text{e}^{d_2}\sin\xi_3)(p_1\text{e}^{r_1}+p_2\text{e}^{r_2}\sin\delta_3+p_3\text{e}^{r_2}\cos\delta_3)\\
&~~~~-k_2(c_1\text{e}^{d_1}+c_3\text{e}^{d_2}\cos\xi_3+c_2\text{e}^{d_2}\sin\xi_3)(p_1\text{e}^{r_1} +p_2\text{e}^{r_2}(\cos\delta_3+\sin\delta_3) +2p_3\text{e}^{r_2}(\cos\delta_3-\sin\delta_3))\bigg]^2,\\
&\Theta_5=-k_1^2(c_1\text{e}^{d_1}+2c_2\text{e}^{d_2}\cos\xi_3-2c_3\text{e}^{d_2}\sin\xi_3)(p_1\text{e}^{r_1}+p_2\text{e}^{r_2}\sin\delta_3+p_3\text{e}^{r_2}\cos\delta_3)\\
&~~~~+k_2^2(c_1\text{e}^{d_1}+c_3\text{e}^{d_2}\cos\xi_3+c_2\text{e}^{d_2}\sin\xi_3)(p_1\text{e}^{r_1}+2p_2\text{e}^{r_2}\cos\delta_3-2p_3\text{e}^{r_2}\sin\delta_3),  \\
&\Theta_6=-k_1(c_1\text{e}^{d_1}+(c_2+c_3)\text{e}^{d_2}\cos\xi_3+(c_2-c_3)\text{e}^{d_2}\sin\xi_3)(p_1\text{e}^{r_1}+p_2\text{e}^{r_2}\sin\delta_3+p_3\text{e}^{r_2}\cos\delta_3)\\
&~~~~+k_2(c_1\text{e}^{d_1}+c_3\text{e}^{d_2}\cos\xi_3+c_2\text{e}^{d_2}\sin\xi_3)(p_1\text{e}^{r_1}+(p_2+p_3)\text{e}^{r_2}\cos\xi_3+(p_2-p_3)\text{e}^{r_2}\sin\xi_3),\\
&\Theta_7=c_1\text{e}^{d_1}+(c_2+c_3)\text{e}^{d_2}\cos\xi_3+(c_2-c_3)\text{e}^{d_2}\sin\xi_3,~
r_1=k_2^3y+(2k_2^3+\frac{3}{2}k_2^2)t,~
r_2=-2k_2^3y-4k_2^3t.
\end{split}
\end{equation}
\section{Conclusions}
The  aim of the present article is to investigate the integrability of the  (2+1)-dimensional damping forcing coupled Burgers equation in the sense of Lax pair. The (2+1)-dimensional damping forcing coupled Burgers equation is obtained by adding damping and forcing terms from couple Burgers equation. Various physical phenomena in fluid dynamics, traffic flow, nonlinear wave can be describle by couple Burgers equation. The damping and forcing terms can exhibit complex behavior in nonlinear couple Burgers equation. The damping and forcing terms has significantly impact on nonlinear evolution equation such as damping term dissipate the energy, stabilize the unstable solution, reduces the wave amplitude as well as forcing term inject energy into the system. The Lax pair in differential form have been constructed of (2+1)-dimensional damping forcing coupled Burgers equation. The expression of  $N$-fold Darboux transformation have been derived by using Lax pair, where $N$ being a positive integer. Using the trivial solution of (2+1)-dimensional damping forcing coupled Burgers equation, one fold and two fold Darboux transformation gives some exact solutions like solitary wave solution and periodic solution.  \\\\
\noindent
\textbf{Acknowledgement:}% Mr. Nanda Kanan Pal is thankful to the University Grants Commission (UGC), India, for providing the financial support under the Junior Research Fellowship program (NTA Ref. No. -211610115390).

\end{document}